\documentclass[12pt]{article}
\usepackage[active]{srcltx}
\usepackage{amssymb,latexsym,amsmath,enumerate,verbatim,amsfonts,amsthm}
\numberwithin{equation}{section}
\newtheorem{theorem}{Theorem}[section]
\newtheorem{corollary}[theorem]{Corollary}
\newtheorem{proposition}[theorem]{Proposition}

\newtheorem{remark}{Remark}[section]

\theoremstyle{definition}

\numberwithin{equation}{section}
\begin{document}
	\title{Infinite and finite dimensional generalized Hilbert tensors	\thanks{This  work was supported by
			the National Natural Science Foundation of P.R. China (Grant No.
			11571095, 11601134).}}

	\date{}
\author{Wei Mei\thanks{School of Mathematics and Information Science,
		Henan Normal University, XinXiang HeNan,  P.R. China, 453007.
		Email: 1017187432@qq.com} , \quad    Yisheng Song\thanks{Corresponding author. School of Mathematics and Information Science  and Henan Engineering Laboratory for Big Data Statistical Analysis and Optimal Control,
		Henan Normal University, XinXiang HeNan,  P.R. China, 453007.
		Email: songyisheng@htu.cn.}}
	\maketitle

	\begin{abstract}
		In this paper, we introduce the concept of an $m$-order $n$-dimensional generalized Hilbert tensor $\mathcal{H}_{n}=(\mathcal{H}_{i_{1}i_{2}\cdots i_{m}})$,
		$$
		\mathcal{H}_{i_{1}i_{2}\cdots i_{m}}=\frac{1}{i_{1}+i_{2}+\cdots i_{m}-m+a},\
		a\in \mathbb{R}\setminus\mathbb{Z}^-;\ i_{1},i_{2},\cdots,i_{m}=1,2,\cdots,n,
		$$
		and show that its $H$-spectral radius and its $Z$-spectral radius are smaller than or equal to $M(a)n^{m-1}$ and $M(a)n^{\frac{m}{2}}$, respectively, here $M(a)$ is a constant only dependent on $a$. Moreover, both infinite and finite dimensional generalized Hilbert tensors are positive definite for $a\geq1$.  For an $m$-order infinite dimensional generalized Hilbert tensor $\mathcal{H}_{\infty}$ with $a>0$, we prove that $\mathcal{H}_{\infty}$ defines a bounded and positively $(m-1)$-homogeneous operator from $l^{1}$ into $l^{p}\ (1<p<\infty)$. The upper bounds of norm of corresponding positively homogeneous operators are obtained.  \vspace{3mm}
		
		\noindent {\bf Key words:}\hspace{2mm} Generalized Hilbert tensor, Spectral radius, Norm, Upper bounds.
	\vspace{3mm}
	
	\noindent {\bf AMS subject classifications (2010):}\hspace{2mm} 47H15, 47H12, 34B10, 47A52, 47J10, 47H09, 15A48, 47H07
	\end{abstract}

	\section{Introduction} Let $\mathbb{R}$ denote the set of all real numbers, and let $\mathbb{Z}$ be the set of all integers. We write $\mathbb{Z}^-$ stands for the set of all non-positive integers, i.e., $$\mathbb{Z}^-=\{k\in \mathbb{Z}: k\leq0\}\mbox{ and }\mathbb{Z}^{--}=\{k\in \mathbb{Z}: k<0\}.$$
An $m$-order $n$-dimensional tensor (hypermatrix) $\mathcal{A} = (a_{i_1\cdots i_m})$ is a multi-array of real entries $a_{i_1\cdots i_m}\in\mathbb{R}$, where $i_j \in I_n=\{1,2,\cdots,n\}$ for $j \in I_m=\{1,2,\cdots,m\}$. $\mathcal{A}$ is
called a {\bf symmetric tensor} if the entries $a_{i_1\cdots i_m}$ are
invariant under any permutation of their indices. Qi \cite{LQ1,LQ2} introduced the concepts of  eigenvalues of the higher order symmetric tensors, and showed the existence of the eigenvalues and some applications.  The concepts of real  eigenvalue was introduced by Lim \cite{LL} independently using a variational approach. Subsequently, many mathematical workers studied the spectral properties of various structured tensors under different  conditions. The spectral properties of nonnegative matrices had been generalized to $n$-dimensional nonnegative tensors under various conditions by Chang et al. \cite{CPZ13,CPT1}, He and Huang \cite{HH2014}, He \cite{H2016}, He et al. \cite{HLK2016}, Li et al. \cite{WLV2015},  Qi \cite{LQ13}, Song and Qi \cite{SQ13,SQ14}, Wang et al. \cite{WZS2016}, Yang and Yang \cite{YY11,YY10} and references therein. \\

For infinite dimensional tensors, the corresponding studies have only just begun, and the conclusions are fewer. Song and Qi \cite{SQ2014} introduced the concept of infinite dimensional Hilbert tensor  and showed that such a Hilbert tensor defines a bounded, continuous and positively $(m-1)$-homogeneous operator from $l^{1}$ into $l^{p}\ (1<p<\infty)$ and the norms of corresponding positively homogeneous operators  are not larger than $\frac{\pi}{\sqrt{6}}$. They also proved that  the spectral radius and $E$-spectral radius of finite dimensional  Hilbert tensor are smaller than or equal to $ n^{m-1}\sin\frac{\pi}{n}$ and $ n^{\frac{m}{2}}\sin\frac{\pi}{n}$, respectively. Clearly, an $m$-order $n$-dimensional  Hilbert tensor  is a Hankel tensor with $v=(1,\frac12, \frac13, \cdots,\frac{1}{nm})$, introduced by  Qi \cite{LQH}. Also see Chen and Qi \cite{CQ2015}, Xu \cite{X2016} for more details of Hankel tensors.  Hilbert tensor (hypermatrix) is a natural extension of  Hilbert matrix, which was introduced by Hilbert \cite{H1894}. For more details of Hilbert matrix, see Frazer\cite{F1946} and Taussky \cite{T1949} for $n$-dimensional Hilbert matrix, Choi \cite{C1983}) and Ingham \cite{I1936} for an
infinite dimensional Hilbert matrix, Magnus \cite{M1950} and Kato \cite{K1957} for  the spectral
properties of infinite dimensional Hilbert matrix.  \\

	In this paper, we study more general Hilbert tensor, which is referred to as ``generalized Hilbert tensor".
	The entries of an $m$-order infinite dimensional generalized Hilbert tensor $\mathcal{H}_{\infty}=(\mathcal{H}_{i_{1}i_{2}\cdots i_{m}})$ are defined by
	\begin{equation}\label{eq11}\mathcal{H}_{i_{1}i_{2}\cdots, i_{m}}=\frac{1}{i_{1}+i_{2}+\cdots +i_{m}-m+a},\  a\in \mathbb{R}\setminus\mathbb{Z}^-,i_{1},i_{2},\cdots i_{m}\in \mathbb{Z}^{++}=-\mathbb{Z}^{--}.\end{equation}
	The entries of an $m$-order $n$-dimensional generalized Hilbert tensor $\mathcal{H}_{n}=(\mathcal{H}_{i_{1}i_{2}\cdots i_{m}})$ are to choose $i_{1},i_{2},\cdots, i_{m}\in \{1,2,\cdots,n\}$  in \eqref{eq11}.
Clearly, both $\mathcal{H}_{n}$ and $\mathcal{H}_{\infty}$ are symmetric and the entries of  the generalized Hilbert tensor with $a\geq1$ can be written as the integral form as follow
	\begin{equation}\label{eq12}
	\mathcal{H}_{i_{1}i_{2}\cdots i_{m}}=\int_{0}^{1}t^{i_{1}+i_{2}+\cdots i_{m}-m+a-1}dt.
	\end{equation}

	For a vector $x=(x_{1},x_{2},\cdots,x_{n})^\top\in \mathbb{R}^n$,  	$ \mathcal{H}_{n}x^{m-1} $ is a vector with its ith component defined by
	\begin{equation}\label{eq13}
	(\mathcal{H}_{n}x^{m-1})_{i}
	=\sum_{i_{2},\cdots,i_{m}=1}^{n}\frac{x_{i_{2}}\cdots x_{i_{m}}}{i+i_{2}+\cdots +i_{m}-m+a},\
	a\in \mathbb{R}\setminus\mathbb{Z}^-,\ i=1,2,\cdots,n.
	\end{equation}
Accordingly, $ \mathcal{H}_{n}x^{m} $ is given by
	\begin{equation}\label{eq14}
	\mathcal{H}_{n}x^{m}=x^\top(\mathcal{H}_{n}x^{m-1}) =\sum_{i_{1},i_{2},\cdots,i_{m}=1}^{n}\frac{x_{i_{1}}x_{i_{2}}\cdots x_{i_{m}}}{i_{1}+i_{2}+\cdots +i_{m}-m+a},a\in \mathbb{R}\setminus\mathbb{Z}^-.
\end{equation}

Let $l^1$ is a space consisting of all real number sequences $x=(x_i)_{i=1}^\infty$ satisfying
$$\sum\limits_{i=1}^{\infty}|x_{i_{i}}|<\infty.$$ 	
	For a real vector $x=(x_{1},x_{2},\cdots,x_{n},x_{n+1},\cdots)\in l^1$,
	$\mathcal{H}_{\infty}x^{m-1}$ is an infinite dimensional vector with its ith component defined by
	\begin{equation}\label{eq15}
	(\mathcal{H}_{\infty}x^{m-1})_{i}
	=\sum_{i_{2},\cdots,i_{m}=1}^{\infty}\frac{x_{i_{2}}\cdots x_{i_{m}}}{i+i_{2}+\cdots +i_{m}-m+a},\
	a\in \mathbb{R}\setminus\mathbb{Z}^-;\ i\in \{1,2,\cdots, n,\cdots\}.
	\end{equation}
Accordingly, $ \mathcal{H}_\infty x^m $ is given by	
		\begin{equation}\label{eq16}
	\mathcal{H}_{\infty}x^{m}
	=\sum_{i_{1},i_{2},\cdots,i_{m}=1}^{\infty}\frac{x_{i_{1}}x_{i_{2}}\cdots x_{i_{m}}}{i_{1}+i_{2}+\cdots +i_{m}-m+a},\
	a\in \mathbb{R}\setminus\mathbb{Z}^- .
\end{equation}

Now we show that both $\mathcal{H}_{\infty}x^{m}$ and $\mathcal{H}_{\infty}x^{m-1}$ are well-defined for all $x\in l^1.$	

	\begin{proposition}\label{pro11}
		Let $\mathcal{H}_{\infty}$ be an $m$-order infinite dimensional generalized Hilbert tensor. Then for all $x\in l^{1}$,\begin{item}
			\item[(i)] $\mathcal{H}_{\infty}x^{m}$ defined by \eqref{eq16} absolutely converges, i.e., $|\mathcal{H}_{\infty}x^{m}|<\infty$;
			\item[(ii)] $\mathcal{H}_{\infty}x^{m-1}$ is well-defined, i.e., for each positive integer $i$, its ith component defined by \eqref{eq15} absolutely converges.
		\end{item}
	\end{proposition}
	
	\begin{proof}
	Let	$[a]$ denote the largest integer  not exceeding $a$. Then for all $a\in \mathbb{R}\setminus\mathbb{Z}^-$,  it is obvious that for all positive integers $i_{1},i_{2},\cdots,i_{m}$,
		$$
	\min_{i_{1},\cdots,i_{m}}	|i_{1}+i_{2}+\cdots +i_{m}-m+a|=a  \mbox{ for }a>0.
		$$
	For $a<0$, there exist some positive integers $i'_{1},i'_{2},\cdots,i'_{m}$ and $i''_{1},i''_{2},\cdots,i''_{m}$ such that $$i'_{1}+i'_{2}+\cdots+i'_{m}-m=-[a] \mbox{ and } i''_{1}+i''_{2}+\cdots+i''_{m}-m=-[a]-1,$$
		and hence, $$
	\min_{i_{1},\cdots,i_{m}}	|i_{1}+i_{2}+\cdots +i_{m}-m+a|=\min\{a-[a],1-(a-[a])\} \mbox{ for }a<0.
		$$
		
	In conclusion,  we have for $ a\in \mathbb{R}\setminus\mathbb{Z}^- $,
		$$
		\frac1{|i_{1}+i_{2}+\cdots +i_{m}-m+a|}\leq N(a)=\begin{cases}
			\frac1a,&a>0;\\
			\frac1{\min\{a-[a],1+[a]-a\}},&a<0.
		\end{cases}$$
		
		Then for $x=(x_{1},x_{2},\cdots,x_{n},x_{n+1},\cdots)\in l^{1}$, we obtain
		$$
		\aligned
		|\mathcal{H}_{\infty}x^{m}|
		&=\left|\sum_{i_{1},i_{2},\cdots,i_{m}=1}^{\infty}\frac{x_{i_{1}}x_{i_{2}}\cdots x_{i_{m}}}{i_{1}+i_{2}+\cdots +i_{m}-m+a}\right|\\
		&\leq \sum_{i_{1},i_{2},\cdots,i_{m}=1}^{\infty}\frac{|x_{i_{1}}x_{i_{2}}\cdots x_{i_{m}}|}{|i_{1}+i_{2}+\cdots +i_{m}-m+a|}\\
		&\leq N(a) \sum_{i_{1},i_{2},\cdots,i_{m}=1}^{\infty}|x_{i_{1}}||x_{i_{2}}|\cdots |x_{i_{m}}|\\
		&=N(a)\left(\sum_{i=1}^{\infty}|x_{i}|\right)^{m}<\infty
		\endaligned
		$$
		since $\sum_{i=1}^{\infty}|x_i|<\infty$, and hence, $\mathcal{H}_{\infty}x^{m}$ absolutely converges.
		
		Similarly, for each positive integer $i$, we also have $$\aligned |(\mathcal{H}_{\infty}x^{m-1})_{i}|
		=&\left|\sum_{i_{2},\cdots,i_{m}=1}^{\infty}\frac{x_{i_{2}}\cdots x_{i_{m}}}{i+i_{2}+\cdots +i_{m}-m+a}\right|\\
		\leq&N(a) \left(\sum_{k=1}^{\infty}|x_k|\right)^{m-1}<\infty. \endaligned$$
		So  $\mathcal{H}_{\infty}x^{m-1}$ is well-defined.
	\end{proof}
	
	In section 2, we present some concepts and basic facts which are used for late. 	In section 3,  we first show that both infinite and finite dimensional generalized Hilbert tensors are positive definite for $a\geq1$. Let \begin{equation}\label{eq17}
	F_{\infty}x=(\mathcal{H}_{\infty}x^{m-1})^{[\frac{1}{m-1}]}
	\mbox{	and }
	T_{\infty}x =
	\begin{cases}
	\|x\|_{l^1}^{2-m}\mathcal{H}_{\infty}x^{m-1}    & x\neq \theta \\
	\theta  &  x=\theta,
	\end{cases}
	\end{equation}
	where $\theta=(0,0,\cdots,0,\cdots)^\top$. Then we  prove that both $F_{\infty}$ and $T_{\infty}$ is a bounded, continuous and positively homogeneous operator on $l^1$. In particular, their  upper bounds with respect to opertor norm are obtained as follows:  \begin{equation}\label{eq18}\|F_\infty\|=\sup_{\|x\|_{l^1}= 1}\|F_\infty x \|_{l^{2(m-1)}}\leq K(a)\mbox{ and }\|T_\infty\|=\sup_{\|x\|_{l^1}= 1}\|T_\infty x \|_{l^2}\leq C(a),\end{equation}
	where $$ K(a)=\begin{cases}
	\sqrt[2(m-1)]{\frac1{a^2}+\frac{\pi^2}{6}}, &0<a<1; \\
	\sqrt[2(m-1)]{\frac{\pi^2}{6}}, &\ a\geq1;
	\end{cases}\mbox{ and } C(a)=\begin{cases}
	\sqrt{\frac1{a^2}+\frac{\pi^2}{6}}, &0<a<1;  \\
	\frac{\pi}{\sqrt{6}}, &\ a\geq1.
	\end{cases}$$
	
For a finite dimensional generalized Hilbert tensor $\mathcal{H}_n$, we show that  its $H$-spectral radius and its $Z$-spectral radius are smaller than or equal to $M(a)n^{m-1}$ and $M(a)n^{\frac{m}{2}}$, respectively, where $$M(a)=\begin{cases}
\frac1a,&a>0;\\
\frac1{\min\{a-[a], 1+[a]-a\}}, &-m(n-1)<a<0;\\
\frac1{-m(n-1)-a}, &a<-m(n-1).
\end{cases}$$ In particular, for $a>0$, we obtain that the upper bounds of its spectral radius $\rho(\mathcal{H}_{n})$ and its $E$-spectral radius $\rho_E(\mathcal{H}_{n})$, i.e.,
\begin{equation}\label{eq19} \rho(\mathcal{H}_{n})\leq \frac{n^{m-1}}a\mbox{ and } \rho_E(\mathcal{H}_{n})\leq \frac{n^{\frac{m}{2}}}a.\end{equation}

	\section{Preliminaries and basic facts}
	
		Let both $X$ and $Y$ be two real Banach spaces with the norm $\|\cdot\|_X$ and $\|\cdot\|_Y$, respectively, and let $T:X\to Y$ be an operator. Then $T$ is said to be
		\begin{itemize}
			\item[(i)]  $t$-homogeneous if $T(\lambda x)=\lambda^tT(x)$ for all $\lambda\in \mathbb{R}$ and all $x\in X$;
			\item[(ii)]  positively homogeneous if $T(tx)=tT(x)$ for all $t>0$ and all $x\in X$;
			\item[(iii)]  bounded if there exits a real number $M>0$ such that
			$$\|Tx\|_{Y}\leq M\|x\|_{X}\mbox{ for all }x\in X.$$
		\end{itemize}
For a bounded, continuous and positively homogeneous operator $T:X\to Y$, the norm of $T$ may be defined as follows (see Fucik et al. \cite{FNSS}, Song and Qi \cite{SQ14} for more details):
	\begin{equation}\label{eq21}
	\|T\|=\sup\{\|Tx\|_{Y}:\|x\|_{X}=1\}.
	\end{equation}

	For $0<p<\infty$, $l^p$ is a space consisting of all real number sequences $x=(x_i)_{i=1}^\infty$ satisfying
$\sum\limits_{i=1}^{\infty}|x_{i_{i}}|^{p}<\infty.$  For $p\geq1$, it is well known  that
$$\|x\|_{l^p}=\left(\sum_{i=1}^{\infty}|x_{i}|^{p}\right)^{\frac{1}{p}}$$
is the norms defined on the sequences space $l^{p}$.	For $p\geq1$, it is well known  that
	$$\|x\|_{p}=\left(\sum_{i=1}^{n}|x_{i}|^{p}\right)^{\frac{1}{p}}$$
	is the norms defined on $\mathbb{R}^n$. Moreover, the following relationship between the two different norms is  obvious,
	\begin{equation}\label{eq22}
	\|x\|_q\leq \|x\|_p\leq n^{\frac{1}{p}-\frac{1}{q}}\|x\|_q\mbox{ for }    q>p.
	\end{equation}
	
	Let $\mathcal{A}$ be an $m$-order $n$-dimensional symmetric tensor. Then a number $\lambda$ is called an eigenvalue of $\mathcal{A}$ if there exists a nonzero vector $x$ such that
	\begin{equation}\label{eq23}
	\mathcal{A}x^{m-1}=\lambda x^{[m-1]}.
	\end{equation}
	where $x^{[m-1]}=(x_{1}^{m-1},x_{2}^{m-1},\cdots x_{n}^{m-1})^\top$, and call $x$ an eigenvector of $\mathcal{A}$ associated with the eigenvalue $\lambda$. We call such an eigenvalue $H$-eigenvalue if it is real and has a real eigenvector $x$, and call such a real eigenvector $x$ an H-eigenvactor. These concepts were first introduced by Qi \cite{LQ1} for  the
	higher order symmetric tensors. Lim \cite{LL} independently
	introduced the notion of eigenvalue for higher order tensors but restricted $x$ to be a real vector and $\lambda$ to be a real number.
	
Qi \cite{LQ1} introduced  another concept of tensor eigenvalue.  	A number $\mu$ is said to be an $E$-eigenvalue of $\mathcal{A}$ if there exists a nonzero vector $x$ such that
	\begin{equation}\label{eq24}
	\mathcal{A}x^{m-1}=\mu x(x^{T}x)^{\frac{m-2}{2}},
	\end{equation}
	and such a nonzero vector $x$ is called an $E$-eigenvector of $\mathcal{A}$ associated with $\mu$. It is clear that if $x$ is real, then $\mu$ is also real. In this case, $\mu$ and $x$ are called a $Z$-eigenvalue of $\mathcal{A}$ and a $Z$-eigenvector of $\mathcal{A}$, respectively.  Qi \cite{LQ1, LQ2} extended some nice properties of symmetric matrices to higher order symmetric tensors. The Perron-Frobenius theorem of nonnegative matrices had been generalized to higher order nonnegative tensors under various conditions by Chang, Pearson and Zhang \cite{CPZ13,CPT1}, Qi \cite{LQ13}, Song and Qi \cite{SQ13,SQ14},  Yang and Yang \cite{YY11,YY10} and references therein.\
	
\section{Generalized Hilbert tensors}
In this section, we mainly discuss the properties of infinite and finite dimensional generalized Hilbert tensors. It is easy to see that both finite and infinite dimensional generalized Hilbert tensor $\mathcal{H}_{n}$ and $\mathcal{H}_{\infty}$ are  symmetric but not always positive semi-definite since $a\in \mathbb{R}\setminus\mathbb{Z}^-$. Howerver, for all $a>0$, both finite and infinite dimensional generalized Hilbert tensors are positive, it follows from the definition of strictly copositive tensors that $\mathcal{H}_{n}$ and $\mathcal{H}_{\infty}$ are strictly copositive, i.e.,
$$\mathcal{H}_{n}x^{m}>0\mbox{ for all nonnegative nonzero vector } x\in \mathbb{R}^{n}, $$
$$\mathcal{H}_{\infty}x^{m}>0\mbox{ for all   nonnegative nonzero vector } x\in l^{1}.$$ The concept of strictly copositive tensors  was introduced by Qi
\cite{LQ13}. Also see Song and Qi \cite{SQ2015} for more details.  Now, we give the positive definitiveness of such two tensors when $a\geq1$.

\subsection{Positive definitiveness}

\begin{theorem}\label{thm31} Let $m$ be even and  $a\geq1$. Then  $m$-order generalized Hilbert tensors  $\mathcal{H}_{\infty}$ and $\mathcal{H}_{n}$ are both positive definite, i.e.,
	$$\mathcal{H}_{\infty}x^{m}>0\mbox{ for all  nonzero vector } x\in l^{1},$$
	$$\mathcal{H}_{n}x^{m}>0\mbox{ for all nonzero vector } x\in \mathbb{R}^n. $$
\end{theorem}
\begin{proof}
	For each $x=(x_{1},x_{2},\cdots,x_{n},x_{n+1},\cdots)\in l^{1}$ and $t\in [0,1]$, it is obvious that $$|t^{i-1+\frac{a-1}{m}}x_{i}|=|t|^{i-1+\frac{a-1}{m}}|x_{i}|\leq |x_i|,$$ which implies that the infinite series $\sum\limits_{i=1}^{\infty}t^{i-1+\frac{a-1}{m}}x_{i}$ uniformly converges with respect to $t\in [0,1]$. Then it follows from integral form of $\mathcal{H}_{i_{1}i_{2}\cdots i_{m}}$ that
	$$
	\aligned
	\mathcal{H}_{\infty}x^{m}&=\sum_{i_2,i_2,\cdots,i_m=1}^{\infty}\int_0^1t^{i_1+i_2+\cdots i_m-m+a-1}x_{i_1}x_{i_2}\cdots x_{i_m}dt\\
	&=\sum_{i_{1},i_{2},\cdots,i_{m}=1}^{\infty}\int_0^1(t^{i_{1}-1+\frac{a-1}{m}}x_{i_{1}})(t^{i_{2}-1+\frac{a-1}{m}}x_{i_{2}})\cdots (t^{i_{m}-1+\frac{a-1}{m}}x_{i_{m}})dt\\
	&=\int_{0}^{1}\left(\sum_{i_{1}=1}^{\infty}t^{i_{1}-1+\frac{a-1}{m}}x_{i_{1}}\right)\left(\sum_{i_2=1}^{\infty}t^{i_2-1+\frac{a-1}{m}}x_{i_2}\right)\cdots \left(\sum_{i_m=1}^{\infty}t^{i_m-1+\frac{a-1}{m}}x_{i_m}\right)dt\\
	&=\int_{0}^{1}\left(\sum_{i=1}^{\infty}t^{i-1+\frac{a-1}{m}}x_i\right)^mdt\geq0\ (m\mbox{ is even}).
	\endaligned
	$$
	
	Now we show $\mathcal{H}_{\infty}x^{m}>0$  for all nonzero vector $ x\in l^{1}$.  Suppose not, then there exists a nonzero vector $x\in l^1$ such that
	$\mathcal{H}_\infty x^{m}=0$, i.e.,
	$$\int_{0}^{1}\left(\sum_{i=1}^\infty t^{i-1+\frac{a-1}{m}}x_{i}\right)^{m}dt=0.$$
	Since $m$ is even, $\left(\sum\limits_{i=1}^\infty t^{i-1+\frac{a-1}{m}}x_{i}\right)^{m}\geq0$, and hence, by the continuity, we have
	$$ t^{\frac{a-1}{m}}\left(x_{1}+\left(\sum\limits_{i=2}^\infty t^{i-1}x_{i}\right)\right)=\sum_{i=1}^\infty t^{i-1+\frac{a-1}{m}}x_{i}=0\mbox{ for all } \ t\in [0,1].$$
	So for all $t\in (0,1]$, we have
	$$x_{1}+\left(\sum\limits_{i=2}^\infty t^{i-1}x_{i}\right)=0.$$
	Since the polynomial $x_{1}+\left(\sum\limits_{i=2}^\infty t^{i-1}x_{i}\right)$ is continuous at $t=0$, we see that
	$$x_{1}+\left(\sum\limits_{i=2}^\infty t^{i-1}x_{i}\right)=0 \mbox{ for all } \ t\in [0,1].$$
	Let $t=0$. Then $x_1=0$, and so, $$t^{\frac{a-1}{m}+1}\left(x_2+\sum\limits_{i=3}^\infty t^{i-2}x_{i}\right)=\sum\limits_{i=2}^\infty t^{i-1+\frac{a-1}{m}}x_{i}=0 \mbox{ for all } \ t\in [0,1].$$
	So, for all $t\in (0,1]$, we have $$x_2+\sum\limits_{i=3}^\infty t^{i-2}x_{i}=0.$$ Again by the continuity, we see that $$x_2+\sum\limits_{i=3}^\infty t^{i-2}x_{i}=0\mbox{ for all } \ t\in [0,1].$$
	Take $t=0$, then we have $x_2=0$.
	
	By the mathematical induction, suppose $x_1=x_2=\cdots=x_k=0$. Then we only need show $x_{k+1}=0$. In fact,
	$$ t^{\frac{a-1}{m}+k}\left(x_{k+1}+\left(\sum\limits_{i=k+2}^\infty t^{i-(k+1)}x_{i}\right)\right)= \sum_{i=k+1}^\infty t^{i-1+\frac{a-1}{m}}x_{i}=0\mbox{ for all } \ t\in [0,1].$$
	So, for all $t\in (0,1]$, we have $$x_{k+1}+\left(\sum\limits_{i=k+2}^\infty t^{i-(k+1)}x_{i}\right)=0.$$ Again by the continuity, we see that $$x_{k+1}+\left(\sum\limits_{i=k+2}^\infty t^{i-(k+1)}x_{i}\right)=0\mbox{ for all } \ t\in [0,1].$$
	Take $t=0$, then we have $x_{k+1}=0$.  So we may conclude that
	$$x_i=0\mbox{ for all } i=1,2,\cdots,n,n+1,\cdots,$$
	and hence, $x=\theta=(0,0,\cdots,0,0,\cdots)^\top$, which is a contradiction. The desired conclusion follows.
	
	Similarly, we can obtain that $\mathcal{H}_n$ is positive definite.
\end{proof}

\begin{remark}
	We show the positive definitiveness of infinite and finite dimensional generalized Hilbert tensors with $a\geq1.$ Then it is unknown  whether or not $\mathcal{H}_n$ and $\mathcal{H}_\infty$ have the positive definitiveness for $0<a<1$ or $a<0$ with $a\in \mathbb{R}\setminus\mathbb{Z}^-$.
\end{remark}
	
	\subsection{Infinite dimensional generalized Hilbert tensors}

	Let
\begin{equation}\label{eq31}
F_{\infty}x=(\mathcal{H}_{\infty}x^{m-1})^{[\frac{1}{m-1}]}\mbox{ ($m$ is even)}
\end{equation}
and
\begin{equation}\label{eq32}
T_{\infty}x =
\begin{cases}
 \|x\|_{l^1}^{2-m}\mathcal{H}_{\infty}x^{m-1}    & x\neq \theta \\
 \theta  &  x=\theta,
\end{cases}
\end{equation}
where $x^{[\frac{1}{m-1}]}=(x_{1}^{\frac{1}{m-1}},x_{2}^{\frac{1}{m-1}},\cdots x_{n}^{\frac{1}{m-1}},\cdots)^\top$ and $\theta=(0,0,\cdots, 0,\cdots)^\top$, the zero element of a vector space $l^p$. It is easy to see that both operators $F_{\infty}$ and $T_{\infty}$ are continuous and positively homogeneous. Therefore our main interests are to study the boundedness of two classes of operators.

	\begin{theorem}\label{thm32}
		Let $F_{\infty}$ and $T_{\infty}$ are defined by Eqs. \eqref{eq31} and \eqref{eq32}, respectively. Assume that  $a>0$.\begin{itemize}
			\item[(i)] If $x\in l^{1}$, then $F_{\infty}x\in l^{p}$ for $m-1<p<\infty$. Moreover, $F_{\infty}$ is a bounded, continuous and positively homogeneous operator from $l^{1}$ into $l^{p}\ (m-1<p<\infty)$. In particular, 	$$\|F_\infty\|=\sup_{\|x\|_{l^1}= 1}\|F_\infty x \|_{l^{2(m-1)}}\leq K(a),$$
			where $$ K(a)=\begin{cases}
			\sqrt[2(m-1)]{\frac1{a^2}+\frac{\pi^2}{6}}, &0<a<1; \\
			\sqrt[2(m-1)]{\frac{\pi^2}{6}}, &\ a\geq1.
			\end{cases}$$
		\item[(ii)] If $x\in l^{1}$, then $T_{\infty}x\in l^{p}$ for $1<p<\infty$. Moreover,  $T_{\infty}$ is a bounded, continuous and positively homogeneous operator from $l^{1}$ into $l^{p}\ (1<p<\infty)$. In particular, $$\|T_\infty\|=\sup_{\|x\|_{l^1}= 1}\|T_\infty x \|_{l^2}\leq C(a),$$
			where $$ C(a)=\begin{cases}
		\sqrt{\frac1{a^2}+\frac{\pi^2}{6}}, &0<a<1; \ \\
		\frac{\pi}{\sqrt{6}}, &\ a\geq1.
		\end{cases}$$
		\end{itemize}
	\end{theorem}
	\begin{proof}
		For $x\in l^{1}$,
		$$
		\aligned
		|(\mathcal{H}_{\infty}x^{m-1})_{i}|
		&=\lim_{n\rightarrow\infty}\left|\sum_{i_{2},\cdots,i_{m}=1}^{n}\frac{x_{i_{2}}\cdots x_{i_{m}}}{i+i_{2}+\cdots +i_{m}-m+a}\right|\\
		&\leq\lim_{n\rightarrow\infty}\sum_{i_{2},\cdots,i_{m}=1}^{n}\frac{|x_{i_{2}}\cdots x_{i_{m}}|}{|i+i_{2}+\cdots +i_{m}-m+a|}\\
		&\leq\lim_{n\rightarrow\infty}\sum_{i_{2},\cdots,i_{m}=1}^{n}\frac{|x_{i_{2}}||x_{i_{3}}|\cdots |x_{i_{m}}|}{|i+\underbrace{1+\cdots +1}_{m-1}-m+a|}\\
		&=\frac{1}{i-1+a}\lim_{n\rightarrow\infty}\sum_{i_{2},\cdots,i_{m}=1}^{n}|x_{i_{2}}||x_{i_{3}}|\cdots |x_{i_{m}}|\\
		&=\frac{1}{i-1+a}\lim_{n\rightarrow\infty}\left(\sum_{k=1}^{n}|x_{k}|\right)^{m-1}\\
		&=\frac{1}{i-1+a}\left(\sum_{k=1}^{\infty}|x_{k}|\right)^{m-1}\\
		&=\frac{1}{i-1+a}\|x\|_{l^1}^{m-1}.
		\endaligned
		$$
	Since $a>0$, then for all positive integer $i>1$, we have $$\frac1{(i-1+a)^s}\leq \frac1{(i-1)^s}.$$
 So the series $\sum\limits_{i=1}^{\infty}\frac1{(i-1+a)^s}$ converges whenever $s>1$.

		(i) For $m-1<p<\infty$, it follows from the definition of $F_{\infty}$ that
		$$
		\aligned
		\sum_{i=1}^{\infty}|(F_{\infty}x)_{i}|^{p}
		&=\sum_{i=1}^{\infty}|\left(\mathcal{H}_{\infty}x^{m-1}\right)_{i}^{\frac{1}{m-1}}|^{p}\\
		&=\sum_{i=1}^{\infty}|\left(\mathcal{H}_{\infty}x^{m-1}\right)_{i}|^{\frac{p}{m-1}}\\
		&\leq\sum_{i=1}^{\infty}\left(\frac{1}{(i-1+a)}\|x\|_{l^1}^{m-1}\right)^{\frac{p}{m-1}}\\
		&=\|x\|_{l^1}^{p}\sum_{i=1}^{\infty}\frac{1}{(i-1+a)^{\frac{p}{m-1}}}<\infty
		\endaligned
		$$	
		since  $s=\frac{p}{m-1}>1$. Thus $F_{\infty}x\in l^{p}$ for all $x\in l^{1}$. Furthermore, we have
		\begin{equation}
		\label{eq33}\|F_{\infty}x\|_{l^p}=\left(\sum_{i=1}^{\infty}|(F_{\infty}x)_{i}|^{p}\right)^{\frac{1}{p}}\leq M\|x\|_{l^1},
		\end{equation}
	where $M=\left(\sum\limits_{i=1}^{\infty}\frac{1}{(i-1+a)^{\frac{p}{m-1}}}\right)^{\frac{1}{p}}$. Therefore $F_{\infty}$ is a bounded, continuous and positively homogeneous operator from $l^{1}$ into $l^{p}\ (m-1<p<\infty)$. Take $p=2(m-1)$, then for $a\geq1$, we have $a-1\geq0$, and hence,
	$$M=\left(\sum\limits_{i=1}^{\infty}\frac{1}{(i-1+a)^2}\right)^{\frac{1}{p}}\leq \left(\sum\limits_{i=1}^{\infty}\frac{1}{i^2}\right)^{\frac{1}{p}}=\sqrt[2(m-1)]{\frac{\pi^2}{6}}.$$
	If $0<a<1$, then
	$$M=\left(\sum\limits_{i=1}^{\infty}\frac{1}{(i-1+a)^2}\right)^{\frac{1}{p}}\leq \left(\frac{1}{(1-1+a)^2}+\sum\limits_{i=2}^{\infty}\frac{1}{(i-1)^2}\right)^{\frac{1}{p}}=\sqrt[2(m-1)]{\frac1{a^2}+\frac{\pi^2}{6}}.$$  It follows from \eqref{eq21} and \eqref{eq33} that
	$$\|F_\infty\|=\sup_{\|x\|_{l^1}= 1}\|F_\infty x \|_{l^{2(m-1)}}\leq M\leq K(a).$$	
	
		(ii) For $ 1<p<\infty$, it follows from the definition of $T_{\infty}$ that
		$$
		\aligned
		\sum_{i=1}^{\infty}|(T_{\infty}x)_{i}|^{p}
		&=\sum_{i=1}^{\infty}|(\|x\|_{l^1}^{2-m}\mathcal{H}_{\infty}x^{m-1})_{i}|^{p}\\
		&=\|x\|_{l^1}^{(2-m)p}\sum_{i=1}^{\infty}|(\mathcal{H}_{\infty}x^{m-1})_{i}|^{p}\\
		&\leq\|x\|_{l^1}^{(2-m)p}\sum_{i=1}^{\infty}\left(\frac{1}{(i-1+a)}\|x\|_{l^1}^{m-1}\right)^{p}\\
		&=\|x\|_{l^1}^{p}\sum_{i=1}^{\infty}\frac{1}{(i-1+a)^p}<\infty
		\endaligned
		$$
	since $s=p>1.$ Thus, $T_{\infty}x\in l^{p}$ for all $x\in l^{1}$, Furthermore, we have
		\begin{equation}
		\label{eq34}\|T_{\infty}x\|_{l^p}=(\sum_{i=1}^{\infty}|(T_{\infty}x)_{i}|^{p})^{\frac{1}{p}}\leq C\|x\|_{l^1},
		\end{equation}
	where $ C=\left(\sum\limits_{i=1}^{\infty}\frac{1}{(i-1+a)^{p}}\right)^{\frac{1}{p}}>0$. So  $T_{\infty}$ is a bounded, continuous and positively homogeneous operator from $l^{1}$ into $l^{p}\ (1<p<\infty)$.  Similarly, take $p=2$, then for $a\geq1$, we have
	$$C=\left(\sum\limits_{i=1}^{\infty}\frac{1}{(i-1+a)^2}\right)^{\frac{1}{2}}\leq\left(\sum\limits_{i=1}^{\infty}\frac{1}{i^2}\right)^{\frac{1}{2}}\leq\sqrt{\frac{\pi^2}{6}}.$$ If $0<a<1$, then
	$$C=\left(\sum\limits_{i=1}^{\infty}\frac{1}{(i-1+a)^2}\right)^{\frac{1}{2}}\leq\left(\frac{1}{(1-1+a)^2}+\sum\limits_{i=2}^{\infty}\frac{1}{(i-1)^2}\right)^{\frac12}=\sqrt{\frac1{a^2}+\frac{\pi^2}{6}}.$$  It follows from \eqref{eq21} and \eqref{eq34} that
	$$\|T_\infty\|=\sup_{\|x\|_{l^1}= 1}\|T_\infty x \|_{l^2}\leq C\leq C(a).$$	The desired results follow.
	\end{proof}
	From the proof of Theorem \ref{thm32}, it is easy to obtain the following properties of operator defined by infinite dimensional generalized Hilbert tensor $\mathcal{H}_{\infty}$.
	\begin{theorem}\label{thm33}
		For an $m$-order infinite dimensional generalized Hilbert tensor $\mathcal{H}_{\infty}$ and $a>0$, let $f(x)=\mathcal{H}_{\infty}x^{m-1}$. Then f is a bounded, continuous and positively $(m-1)$-homogeneous operator from $l^{1}$ into $l^{p}\ (1<p<\infty)$.
	\end{theorem}

	\begin{remark}\begin{itemize}
			\item[(i)] In Theorem \ref{thm33}, $\mathcal{H}_{\infty}$ defines a bounded operator for $a>0$. Then it is unknown whether or not $\mathcal{H}_{\infty}$ is  bounded for $a<0$ with $a\in \mathbb{R}\setminus\mathbb{Z}^-$.
			\item[(ii)] For an $m$-order infinite dimensional generalized Hilbert tensor, the upper bounds of norm of corresponding positively homogeneous operators are showed for $a>0$, then for $a<0$   with $a\in \mathbb{R}\setminus\mathbb{Z}^-$, it is unknown whether  have similar conclusions or not.
			\item[(iii)] Are the upper bounds  the best in Theorem \ref{thm32}?
		\end{itemize}
		
	\end{remark}
	
	\subsection{Finite dimensional generalized Hilbert tensors}

	\begin{theorem}\label{thm34}
		Let $\mathcal{H}_{n}$ be an $m$-order $n$-dimensional generalized Hilbert tensor. Assume that $a\in \mathbb{R}\setminus \mathbb{Z}^-$  and
		$$M(a)=\begin{cases}
		\frac1a,&a>0;\\
		\frac1{\min\{a-[a], 1+[a]-a\}}, &-m(n-1)<a<0;\\
		\frac1{-m(n-1)-a}, &a<-m(n-1),
		\end{cases}$$ where $[a]$ denotes the largest integer not exceeding $a$. Then
		\begin{itemize}
			\item[(i)] $|\lambda|\leq M(a)n^{m-1}$ for all $H$-eigenvalues $\lambda$ of generalized Hilbert tensor $\mathcal{H}_{n}$ if $m$ is even;
			\item[(ii)] $|\mu|\leq  M(a)n^{\frac{m}{2}}$ for all $Z$-eigenvalues $\mu$ of generalized Hilbert tensor $\mathcal{H}_{n}$.
		\end{itemize}	
	\end{theorem}
	\begin{proof} Since $a\in \mathbb{R}\setminus \mathbb{Z}^-$, Using the proof technique of Proposition \ref{pro11}, it is obvious that for all $i_{1},i_{2},\cdots,i_{m}\in\{1,2,\cdots, n\}$,
		$$\aligned |i_{1}+i_{2}+\cdots +i_{m}-m+a|&\geq a\mbox{ for } a>0,\\
		|i_{1}+i_{2}+\cdots +i_{m}-m+a|&\geq \min\{a-[a],1-(a-[a])\} \mbox{ for } -m(n-1)<a<0,\\
		|i_{1}+i_{2}+\cdots +i_{m}-m+a|&\geq-m(n-1)-a \mbox{ for } a<-m(n-1).
		\endaligned$$
	In conclusion, $$\frac1{|i_{1}+i_{2}+\cdots +i_{m}-m+a|}\leq M(a).$$
	Then	for each nonzero vector $x\in \mathbb{R}^n$, we have
		$$
		\aligned
		|\mathcal{H}_{n}x^{m}|&=\left|\sum_{i_{1},i_{2},\cdots,i_{m}=1}^{n}\frac{x_{i_{1}}x_{i_{2}}\cdots x_{i_{m}}}{i_{1}+i_{2}+\cdots +i_{m}-m+a}\right|\\
		&\leq\sum_{i_{1},i_{2},\cdots,i_{m}=1}^{n}\frac{|x_{i_{1}}x_{i_{2}}\cdots x_{i_{m}}|}{|i_{1}+i_{2}+\cdots +i_{m}-m+a|}\\
		&\leq M(a)\sum_{i_{1},i_{2},\cdots,i_{m}=1}^{n}|x_{i_{1}}||x_{i_{2}}|\cdots,|x_{i_{m}}|\\
		&=M(a)(\sum_{i=1}^{n}|x_{i}|)^{m}\\
		&=M(a)\|x\|_1^{m}.
		\endaligned
		$$
		(i) From the fact that $\|x\|_{1}\leq n^{1-\frac{1}{m}}\|x\|_{m}$, it follows that
		$$
		\aligned
		|\mathcal{H}_{n}x^{m}|&\leq M(a)\|x\|_{1}^{m}\leq M(a)(n^{1-\frac{1}{m}}\|x\|_{m})^{m}\\
		&=M(a)n^{m-1}\|x\|_{m}^{m},
		\endaligned
		$$
		and hence, for all  nonzero vector $x\in \mathbb{R}^n$, we have
	\begin{equation}\label{eq35}	\left|\mathcal{H}_{n}\left(\frac{x}{\|x\|_{m}}\right)^{m}\right|= \frac{|\mathcal{H}_{n}x^m|}{\|x\|_{m}^{m}}
		\leq M(a)n^{m-1}.
		\end{equation}
		It follows from the definition of eigenvalues of tensor that for each $H$-eigenvalue  $\lambda$ of generalized Hilbert tensor $\mathcal{H}_{n}$, there exists a   nonzero vector $y\in \mathbb{R}^n$ such that $$\mathcal{H}_{n}y^{m-1}=\lambda y^{[m-1]},$$
		and so, $$\mathcal{H}_{n}y^m=y^\top(\mathcal{H}_{n}y^{m-1})=\lambda y^\top y^{[m-1]}=\lambda\sum_{i=1}^{n}y_i^m=\lambda\|y\|_m^m$$
		since $m$ is even. Thus, it follows from \eqref{eq35} that $$
		|\lambda|=\frac{|\mathcal{H}_{n}y^m|}{\|y\|^m_m}=\left|\mathcal{H}_{n}\left(\frac{y}{\|y\|_{m}}\right)^m\right|\leq M(a)n^{m-1}.$$
		Since $\lambda$ is arbitrary $H$-eigenvalue, this obtains the conclusion (i).
		
		(ii) Similarly, from the fact that $\|x\|_{1}\leq \sqrt{n}\|x\|_{2}$, it follows that
		$$|\mathcal{H}_{n}x^{m}|\leq M(a)\|x\|_{1}^{m}
		\leq M(a) n^{\frac{m}{2}}\|x\|_{2}^{m},
		$$
		and hence, for all nonzero vector $x\in \mathbb{R}^n$, we have\begin{equation}\label{eq36}
		\left|\mathcal{H}_{n}\left(\frac{x}{\|x\|_{2}}\right)^{m}\right|=\frac{|\mathcal{H}_{n}x^m|}{\|x\|_{2}^m}
		\leq M(a)n^{\frac{m}{2}}.
		\end{equation}
		It follows from the definition of $Z$-eigenvalues of tensor that for each $Z$-eigenvalue  $\mu$ of generalized Hilbert tensor $\mathcal{H}_{n}$, there exists a nonzero vector $z\in \mathbb{R}^n$ such that $$\mathcal{H}_{n}z^{m-1}=\mu z(z^\top z)^{\frac{m-2}{2}},$$
		and so, $$\mathcal{H}_{n}z^m=z^\top(\mathcal{H}_{n}z^{m-1})=\mu z^\top z(z^\top z)^{\frac{m-2}{2}}=\mu\left(\sum_{i=1}^{n}z_i^2\right)^{\frac{m}2}=\mu\|z\|_2^m.$$
	 Thus, by \eqref{eq36}, we have
		$$|\mu|=\frac{|\mathcal{H}_{n}z^m|}{\|z\|_2^m}=\left|\mathcal{H}_{n}\left(\frac{z}{\|z\|_2}\right)^m\right|\leq M(a)n^{\frac{m}2}.$$
		Since $\mu$ is arbitrary $Z$-eigenvalue, the desired conclusion follows.
	\end{proof}

Let  $a>0$. Then the generalized Hilbert tensor $\mathcal{H}_{n}$ is positive (all entries are positive) and symmetric, and  hence, its spectral radius $\rho(\mathcal{H}_{n})$ and its $E$-spectral radius $\rho_E(\mathcal{H}_{n})$ is its $H$-eigenvalue and its $Z$-eigenvalue, respectively. For more details, see Chang, Pearson and Zhang \cite{CPZ13,CPT1}, Qi \cite{LQ13}, Song and Qi \cite{SQ13,SQ14} and Yang and Yang \cite{YY11,YY10}.  So the following colusions are easy to obtain by Theorem \ref{thm34}.

\begin{corollary} \label{cor35}
Let $\mathcal{H}_{n}$ be an $m$-order $n$-dimensional generalized Hilbert tensor. Assume that $a>0$. Then
\begin{itemize}
	\item[(i)] $|\lambda|\leq \rho(\mathcal{H}_{n})\leq \frac{n^{m-1}}a$ for all eigenvalues $\lambda$ of generalized Hilbert tensor $\mathcal{H}_{n}$ if $m$ is even;
	\item[(ii)] $|\mu|\leq \rho_E(\mathcal{H}_{n})\leq \frac{n^{\frac{m}{2}}}a$ for all $E$-eigenvalues $\mu$ of generalized Hilbert tensor $\mathcal{H}_{n}$.
\end{itemize}	
\end{corollary}
	
	\begin{remark}
		 We prove  the upper bounds of  spectral radii in the above results.  These upper bounds may not be the best, then the following problems deserve us further research:
		
		  What is the best upper bounds? How compute such an upper bound?
	\end{remark}
	

\end{document}